\documentclass[12pt]{amsart}
\usepackage{amssymb}
\usepackage{hyperref}

\usepackage[color, notref, notcite]{showkeys}     
\definecolor{refkey}{gray}{.5}   
\definecolor{labelkey}{gray}{.5} 
\usepackage{graphicx}
\usepackage{epsfig}

\theoremstyle{plain}
\newtheorem{conjecture}{Conjecture}
\newtheorem{theorem}{Theorem}
\newtheorem{claim}{Claim}

\newtheorem{proposition}{Proposition}
\newtheorem{corollary}{Corollary}
\theoremstyle{definition}

\newtheorem{remark}[theorem]{Remark}

\newcommand{\const}{\text{const}}
\numberwithin{equation}{section}

\newcommand{\R}{\mathbb{R}}

\renewcommand{\phi}{\varphi}

\newcommand{\authorfootnotes}{\renewcommand\thefootnote{*}}%
\usepackage{amsmath}
\DeclareMathOperator{\sgn}{sgn}

\begin{document}

\input epsf.sty

\title{Hausdorff dimension of a class of three-interval exchange maps}

\maketitle{}

\begin{center}
\authorfootnotes
 D. Karagulyan\footnote{Department of Mathematics, Royal Institute of Technology, S-100 44 Stockholm, Sweden. Email: davitk@kth.se} \par \bigskip
\end{center}

\begin{abstract}
In \cite{B} Bourgain proves that Sarnak's disjointness conjecture holds for a certain class of Three-interval exchange maps. In the present paper we slightly improve the Diophantine condition of Bourgain and estimate the constants in the proof. We further show, that the new parameter set has positive, but not full Hausdorff dimension. This, in particular, implies that the Lebesgue measure of this set is zero. 

\end{abstract}

\section{Introduction}\label{intro}

Let $\mu$ denote the M{\"o}bius function, i.e.
$$
\mu(n) = \begin{cases} 
(-1)^k & \mbox{ if } n=p_1 p_2 \cdots p_k \mbox{ for distinct primes } p_k, \\
0& \mbox{ otherwise}.\\
\end{cases}
$$
In \cite{Sa1}, \cite{Sar} Sarnak introduced the following conjecture. Recall
that a topological dynamical system $(Y,T)$ is a compact metric
space $Y$ with a homeomorphism $T : Y \rightarrow Y$ , and the topological entropy
$h(Y,T)$ of such a system is defined as
$$
h(Y,T)=\lim_{\epsilon\rightarrow 0}\limsup_{n\rightarrow \infty}\frac{1}{n}\log N(\epsilon,n),
$$
where $N(\epsilon, n)$ is the largest number of $\epsilon$-separated points in $Y$ using the metric $d_n:Y\times Y \rightarrow \R^+$ defined by 
$$
d_n(x,y)=\max_{0\leq i\leq n}d(T^ix,T^iy).
$$
A sequence $f:\mathbb{Z} \rightarrow \mathbb{C}$ is said to be deterministic if it is of the form
$$
f(n)=F(T^nx),
$$
for all $n$ and some topological dynamical system $(Y,T)$ with zero topological entropy $h(Y,T)=0$, a base  point $x \in Y$, and a continuous function $F:Y \rightarrow \mathbb{C}$.
\begin{conjecture}[Sarnak]\label{con1}
Let $f:\mathbb{N}\rightarrow \mathbb{C}$ be a deterministic  sequence. Then 
\begin{equation} \label{f0}
S_n(T(x),f) = \frac{1}{n} \sum_{k=1}^n \mu(k)f(k)=o(1),
\end{equation}
as $n \rightarrow \infty$.
\end{conjecture}

The conjecture, also known as the M{\"o}bius orthogonality or M{\"o}bius disjointness conjecture, is known to be true for several dynamical systems.
Note, that in the simplest case, when $f \equiv \const$, the conjecture is equivalent to the statement
$$
\frac{1}{N}\sum_{n=1}^{N}\mu(n)=o(1),
$$
which, in fact, is equivalent to the Prime Number Theorem.
The orthogonality of the M{\"o}bius function to any sequence arising from a rotation
dynamical system ($X$ is the circle $\mathbb{T}$ and $T(x) = x + \alpha$, $\alpha \in \mathbb{T}$) follows from the
following inequality of Davenport(\cite{davenport})
$$
\max_{\theta \in T} \left| \sum_{k \leq x}\mu(k)e^{ik\theta}\right|\leq C_A\frac{x}{\log^Ax },
$$
for any $A>0$. However this result predates Sarnak's conjecture and the methods used in the proof are number-theoretical. When
$(X, T)$ is a translation on a compact nilmanifold it is proved in \cite{G-T}. In \cite{B-S-Z} it is established also for the discrete horocycle flows. For orientation preserving circle-homeomorphisms and continuous interval maps of zero entropy the conjectures is proved in \cite{D}. The conjecture has also been proved to hold in several other cases (\cite{A-K-L-1},\cite{A-L},\cite{AKL}).
Another natural class of dynamical systems are the interval exchange maps. In \cite{B} Bourgain, using the dynamical description of trajectories in (\cite{F-H-Z-1}--\cite{F-H-Z-3}) and the Hardy-Littlewood circle method, showed that under a certain diophantine condition Conjecture \ref{con1} holds for a certain class of three-interval exchange maps. In this paper we slightly improve the diophantine condition of Bourgain and estimate the Hausdorff dimension of the new parameter set (Theorem \ref{mainth}).
We want to note, that using the criterion of Bourgain in  \cite{B} and the generalization of the self-dual induction defined in \cite{F}, for each primitive permutation, Ferenczi and Mauduit(\cite{F-M}) construct a large family of $k$-interval exchanges
satisfying Sarnak’s conjecture.

In \cite{A-Ch} Eskin and Chaika proved the M{\"o}bius orthogonality for three interval exchange maps satisfying a certain mild diophantine condition. Even though their result holds for almost all three interval exchange maps, the diophantine condition considered in their paper is essentially complementary to the one considered here. In \cite{A-Ch} the continued fractions are required to have certain bound from above, while in Bourgains method they need to be uniformly large. We note that Eskin and Chaika, in fact, give two proofs of the fact that the M{\"o}bius orthogonality holds almost surely for three-interval exchange maps, however the second proof does not provide an explicit Diophintine condition. Their proof is based on the Katai \cite{Kat} and Bourgain-Sarnak-Ziegler \cite{B-S-Z} criterion, while Bourgain uses a direct approach.

The present paper is a part of the author’s Ph.D. thesis.

\section{Three interval exchange maps}\label{sec:2}

The three-interval exchange transformation $T$ with probability vector $(\alpha,\beta,1 - (\alpha + \beta))$, $ 0 < \alpha < 1$, $0 < \beta < 1 - \alpha$, and the permutation $(3,2,1)$ is defined by 
$$
T_{\alpha,\beta}(x)=\begin{cases} 
x + 1 - \alpha   & \mbox{ if } x\in [0,\alpha), \\
x + 1 - 2\alpha -\beta  & \mbox{ if } x\in [\alpha,\alpha + \beta), \\
x - \alpha -\beta  & \mbox{ if } x\in [\alpha + \beta,1). \\
\end{cases}
$$
$T$ depends only on the two parameters $(\alpha, \beta)$. We
note that $T$ is continuous except at the points $\alpha$ and $\alpha + \beta$.

In order to present the main result of the paper we need to recall some facts and definitions from  \cite{F-H-Z-1} and \cite{B}.
Set
\begin{equation}\label{func}
A(\alpha, \beta) = \frac{1 - \alpha}{1 + \beta}
\end{equation}
and 
$$
B(\alpha, \beta) = \frac{1}{1 + \beta}.
$$
$T$ is obtained from the $2$-interval exchange map $R$ on $[0,1)$ given by (\cite{K},\cite{K-S})
\begin{equation}\label{ind}
R(x)=\begin{cases} 
x + A(\alpha,\beta),   & \mbox{ if } x\in [0,\alpha), \\
x + A(\alpha, \beta)-1  & \mbox{ if } x\in [1-A(\alpha,\beta),1), \\
\end{cases}
\end{equation}
by inducing (according to the first return map) on the subinterval $[0, B(\alpha,\beta)]$ and then renormalizing 
by scaling by $1+\beta$.
We say $T$ satisfies the infinite distinct orbit condition (or i.d.o.c. for short) of Keane \cite{K} if the two negative trajectories $\{T^{-n}(\alpha)\}_{n\geq 0}$ and $\{T^{-n}(\alpha + \beta)\}_{n\geq 0}$ of the discontinuities are infinite disjoint sets. Under this hypothesis, $T$ is both minimal and uniquely ergodic; the unique invariant probability measure is the Lebesgue measure $\mu$ on $[0,1)$ (and hence $(X,T,\mu)$ is an ergodic system).

Let $I$ denote the open interval $(0,1)$, $D_0 \subset \mathbb{R}^2$, the simplex bounded by the lines $y = 0$, $x = 0$, and $x + y=1$, and $D$ the triangular region bounded by the lines $x = \frac{1}{2}$, $x + y = 1$, and $2x + y = 1$. Note that
$$
D_0=\{(\alpha,\beta): 0 < \alpha < 1, 0 < \beta < 1 - \alpha\}.
$$
We define two mappings on $I\times I$
$$
F(x,y) =\left(\frac{2x-1}{x},
\frac{y}{x}\right) \hbox{ and }  G(x,y) = (1-x-y,y). 
$$
According to \cite{F-H-Z-1}, if $(\alpha,\beta) \in D_0$ is not in $D$ and is not on any of the rational lines $p\alpha + q\beta = p-q$, $p\alpha+q\beta = p-q +1$, $p\alpha+q\beta = p-q-1$, then there exists a unique finite sequence of integers $l_0, l_1, ..., l_k$ such that $(\alpha,\beta)$ is in $H^{-1}D$, where $H$ is a composition of the form $G^t \circ F^{l_0} \circ G \circ F^{l_1} \circ G \dots \circ G \circ F^{l_k} \circ G^s, s,t \in \{0,1\}$.
Let 
$$
\mathcal{H}=\{ G^t \circ F^{l_0} \circ G \circ F^{l_1} \circ G \dots \circ G \circ F^{l_k} \circ G^s: s,t \in \{0,1\} \hbox{ and } l_0, l_1, \dots, l_k \in \mathbb{N} \}.
$$
Clearly $\mathcal{H}$ is a countable set.

The function $H(\alpha,\beta)$ is computed recursively as follows: we start with $\alpha^{(0)} = \alpha$, $\beta^{(0)} = \beta$. Then, given $(\alpha^{(k)},\beta^{(k)})$, we have three mutually exclusive possibilities: if $(\alpha^{(k)},\beta^{(k)})$ is in $D$, the algorithm stops; if $\alpha^{(k)} < \frac{1}{2}$, we apply $G$; if $2\alpha^{(k)} + \beta^{(k)} < 1$, we apply $F$.  

Associated to each point $(\alpha,\beta) \in D_0$, there is a sequence $(n_k,m_k,\varepsilon_{k+1})_{k\geq 1}$, where $n_k$ and $m_k$ are positive integers, and $\varepsilon_{k+1}=\pm1$. This sequence is called the three-interval expansion of $(\alpha,\beta)$; it is constructed as follows:
\begin{itemize}
  \item For $(\alpha,\beta)$ in $D$ let
  $$
  x_0 =
\frac{1-\alpha-\beta}{1-\alpha} \hbox{ and }
y_0 = \frac{1-2\alpha}{1-\alpha}, 
$$
and define for $k \geq 0$ 
$$
(x_{k+1},y_{k+1})=\begin{cases} 
\left(\left\{\frac{y_k}{(x_k + y_k) - 1},\frac{x_k}{(x_k + y_k) - 1}\right\}\right) & \mbox{ if } x_k+y_k > 1 \\
\left(\left\{\frac{1 - y_k}{1 -(x_k + y_k)},\frac{1 - x_k}{1 - (x_k + y_k)}\right\}\right) & \mbox{ if } x_k+y_k < 1.\\
\end{cases}
$$
$$
(n_{k+1}, m_{k+1})=\begin{cases} 
\left(  \left[  \frac{y_k}{(x_k + y_k) - 1} \right] ,\left[ \frac{x_k}{(x_k + y_k) - 1}\right]\right)   & \mbox{ if } x_k+y_k > 1 \\
\left( \left[ \frac{1 - y_k}{1 -(x_k + y_k)} \right]  ,\left[ \frac{1 - x_k}{1 - (x_k + y_k)}\right] \right)  & \mbox{ if } x_k+y_k < 1,\\
\end{cases}
$$
where $\{a\}$ and $[a]$ denote the fractional and integer part of $a$ respectively. For $k \geq 0$ set 
$$
\epsilon_{k+1} = \sgn(x_k + y_k - 1). 
$$
We note that $\epsilon_1$ is always $-1$, hence we ignore it in the expansion.
  \item For $(\alpha,\beta)\notin D$ we let $H$ be the function above for which $(\alpha,\beta)\in H^{-1}D$ and put
  $$
  (\bar{\alpha},\bar{\beta})=H(\alpha,\beta),
  $$
and define $(n_k,m_k,\epsilon_{k+1})$ as in the previous case, starting from $(\bar{\alpha},\bar{\beta})\in D$.
\end{itemize}

In \cite{F-H-Z-1} the authors also prove the following propositions and theorem:

\begin{proposition}[{\cite[Proposition 2.1, (2)]{F-H-Z-1}}]\label{propos1}
An infinite sequence $(n_k, m_k, \epsilon_{k+1})$ is the expansion of at least one pair $(\alpha, \beta)$ defining a transformation $T$
satisfying the i.d.o.c. condition, if and only if $n_k$ and $m_k$ are positive integers, $\epsilon_{k+1} = \pm1$, $(n_k, \epsilon_{k+1}) \neq (1, +1)$ and $(m_k, \epsilon_{k+1}) \neq (1, +1)$ for infinitely many values of $k$.

\end{proposition}

\begin{proposition}[{\cite[Proposition 2.1, (4)]{F-H-Z-1}}]\label{propos}
For $(\alpha,\beta)\in D_0$, let $(\bar{\alpha},\bar{\beta})=H(\alpha,\beta)$ as above, then
$$
A(\bar{\alpha},\bar{\beta})= \frac{1 - \bar{\alpha}}{1 + \bar{\beta}} = \cfrac{1}{2 + \cfrac{1}{m_1  + n_1 - \cfrac{\epsilon_2}{m_2 + n_2 - \cfrac{\epsilon_3}{m_3 + n_3 - \ddots\,}}}}.
$$
\end{proposition}

\medskip
We define the natural partition
$$
P_1=[0,\alpha),
$$
$$
P_2=[\alpha, \alpha+\beta),
$$
$$
P_3=[\alpha+\beta,1).
$$
For every point $x \in [0,1)$, we define an infinite sequence $(x_n)_{n\in \mathbb{N}}$ by putting $x_n=i$ if $T^nx \in P_i$, $i=1,2,3$. The sequence $(x_1,x_2,\dots)$ is called the trajectory of $x$. If $T$ satisfies the i.d.o.c. condition (see \cite{K}), the minimality of the system implies that all trajectories contain the same finite words as factors. 

Let $I'$ be a set of the form $\cap_{i=0}^{n-1}T^{-i}P_k$; we say $I'$ has a name  of length $n$ given by $k_0,\dots, k_{n-1}$; note that $I'$ is necessarily an interval and $k_0,\dots, k_{n-1}$ is the common  beginning of trajectories of all points in $I'$.

For each interval $J$, there exists a partition $J_i$, $1 \leq i \leq t$, of $J$ into
subintervals (with $t = 3$ or $t = 4$), and $t$ integers $h_i$, such that $T^{h_i}J_i \subset J$, and $\{T^jJ_i\}$, $1\leq i \leq t$, $0\leq j \leq h_i - 1$, is a partition of $[0,1)$ into intervals: this is the partition into Rokhlin stacks associated to $T$ with respect to $J$. The intervals $J_i$ have names of length $h_i$ and are called return words to $J$.
 
We have the following theorem.

\begin{theorem}[{\cite[Theorem 2.2]{F-H-Z-1}}]\label{Ft}
Let $T$ satisfies the i.d.o.c. condition, and let 
$$
(n_k,m_k,\epsilon_{k+1})_{k\geq1},
$$
be the three-interval expansion of $(\alpha,\beta)$. Then there exists an infinite sequence of nested intervals $J_k$, $k \geq 1$, which have exactly three return words, $A_k$, $B_k$ and $C_k$, given recursively for $k \geq 1$ by the following formulas 
\begin{equation}\label{form1}
A_k = A^{n_k-1}_{k-1} C_{k-1}B^{m_k-1}_{k-1} A_{k-1},
\end{equation}
\begin{equation}\label{form2}
B_k = A^{n_k-1}_{k-1} C_{k-1}B^{m_k}_{k-1},
\end{equation}
\begin{equation}\label{form3}
C_k = A^{n_k-1}_{k-1} C_{k-1}B^{m_k-1}_{k-1},
\end{equation}
if $\epsilon_{k+1}=+1$, and 
\begin{equation}\label{form4}
A_k = A^{n_k-1}_{k-1} C_{k-1}B^{m_k}_{k-1},
\end{equation}
\begin{equation}\label{form5}
B_k = A^{n_k-1}_{k-1} C_{k-1}B^{m_k-1}_{k-1} A_{k-1},
\end{equation}
\begin{equation}\label{form6}
C_k = A^{n_k-1}_{k-1} C_{k-1}B^{m_k}_{k-1}A_{k-1},
\end{equation}
if $\epsilon_{k+1}=-1$.
The initial words $A_0$,$B_0$,$C_0$ satisfy $||A_0|-|B_0||= 1$ and they are simple combinations of the symbols $1,2$ and $3$ (see Proposition 2.3, \cite{F-H-Z-1}).
\end{theorem}
Let $a_k = |A_k|$, $b_k = |B_k|$, $c_k = |C_k|$. Note that $|a_k-b_k| = |a_{k-1}-b_{k-1}| = 1$. It follows from Theorem \ref{Ft} and Proposition 2.3 in \cite{F-H-Z-1}, that
 $|a_k - b_k| = |a_0 - b_0| = 1$ and $c_k$ is either $a_k - a_{k-1} = b_k - b_{k-1}$ or $a_k + a_{k-1} = b_k + b_{k-1}$. In
particular, we have 
$c_k \leq 2a_k$ and $c_k \leq 2b_k$.

\begin{remark} In \cite{A-Ch} Eskin and Chaika show, that for three interval exchange maps satisfying the conditions $(A0)$--$(A9)$ (page 3), Sarnak's conjecture holds. We want to note, that their approach is also based on the fact that three interval exchange maps can be induced from a two interval exchange map. The numbers $\{a_k \}_{k=1}^{\infty}$, in conditions $(A0)$--$(A9)$, are the continued fractions of the rotation number of the two interval exchange map which induces the three interval exchange map with parameters $(\alpha,\beta)$. In our case this corresponds to the number $A(\alpha,\beta)$ (see \eqref{ind}), hence the continued fractions of $A(\alpha,\beta)$ are the numbers $\{a_k\}_{k=1}^\infty$ and from Proposition \ref{propos} it is easy to see, that they are related to the numbers $\{(n_k+m_k)\}_{k=1}^\infty$. However in Bourgain's approach this numbers are required to be sufficiently large (see Theorem \ref{th1}), while the conditions $(A0)$--$(A9)$ essentially give upper bounds. 
\end{remark}

Using the dynamical descriptions of trajectories in \cite{F-H-Z-1}, Bourgain \cite{B}, proves  Sarnak's disjointness conjectures for a certain class of three interval exchange maps. Now we recall the statement of his theorem.

Consider a symbolic system on the alphabet $V$ with finitely many symbols and with order-$n$ words $W \in \mathcal{W}_n$ of the form
\begin{equation}\label{f11}
W = W^{k_1}_1 W^{k_2}_2 \cdots W_r^{k_r} \hbox{ for some } W_1,W_2, \dots ,W_r \in \mathcal{W'}_{n-1}=\bigcup_{m<n}\mathcal{W}_{m},
\end{equation}
where it is assumed that $r$ remains uniformly bounded, $r < C$. 
It is also assumed the following property for the system $\{W_n\}$. For $W \in \mathcal{W}_n$, which is expressed in words $W' \in \mathcal{W}_{n-s}$, $0<s\leq n$, by iteration of \eqref{f11} we have, 
\begin{equation}\label{f12}
\frac{|W|}{\max|W'|}>\beta(s),
\end{equation}
where 
\begin{equation}\label{f13}
\beta(s)>C_0^{s},
\end{equation}
for some $s$ and sufficiently large constant $C_0$.
 
\begin{theorem}[{\cite[Theorem 2, page 126]{B}}]\label{thm}
Let $\{W_n;n \geq 1\}$ be a symbolic system with properties \eqref{f11}--\eqref{f13} and $\sigma$ be the shift on the system. Then, if $W \in \bigcup \mathcal{W}_n$ and $|W| = N$, one has 
$$
\int_\mathbb{T} |P_W(\theta)||\sum_{k=1}^N \mu(k)e(k\theta)|d\theta = \mathcal{O}_A(N(\log N)^{-A}),
$$
for any $A>0$, where
$$
P_W(\theta)=\sum_{k=1}^N f(k) e({k\theta}),
$$
and $f(k)=f(\sigma^k(x))$.
\end{theorem}
To see how this implies Sarnak's conjecture, we recall the following inequality, which immediately follows from Parseval's identity
\begin{equation}\label{ident}
\left|\sum_{k=1}^N \mu(k)f(k)\right|\leq \int_\mathbb{T} |\sum_{k=1}^N f(k) e({k\theta})||\sum_{k=1}^N \mu(k)e(-k\theta)|d\theta.
\end{equation}
From here and Theorem \ref{thm}
\begin{equation}\label{ident22}
\left|\sum_{k=1}^N \mu(k)f(k)\right|\leq \int_\mathbb{T} |P_W(\theta)||\sum_{k=1}^N \mu(k)e(-k\theta)|d\theta\leq \mathcal{O}_A(N(\log N)^{-A}),
\end{equation}
which implies Sarnak's conjecture.
For three interval exchange maps we have
$$
\mathcal{W}_k = \{ A_k,B_k,C_k \}, \hbox{ for } k \geq 1.
$$
One can see, that if $m_k$ and $n_k$ are uniformly large, then the conditions \eqref{f12}--\eqref{f13} are satisfied and as a corollary from Theorem \ref{thm} one gets the following result:
\begin{theorem}[{\cite[Theorem 3]{B}}]\label{th1}
Assume $T_{\alpha,\beta}$ is a three-interval exchange transformation satisfying the Keane condition and such that the associated three-interval expansion sequence 
$$
(n_k,m_k,\epsilon_{k+1})_{k\geq1}
$$
of integers fulfills the conditions
\begin{equation}\label{bou1}
\min (n_k,m_k)\geq C_0,  \hbox{ for } k\geq k_{\alpha,\beta},
\end{equation}
for $C_0$ sufficiently large.
Then $T_{\alpha,\beta}$ satisfies Sarnak's disjointness conjecture.
\end{theorem}

\begin{remark}\label{rmk}
Note, that Bourgain's theorem is in fact more general than the form it is stated in Theorem \ref{th1}. As it was mentioned above, in Theorem \ref{th1} it is assumed, that there is uniform expansion at each steps, i.e. \eqref{bou1}, but as one can see from the conditions \eqref{f12}--\eqref{f13} it is sufficient to have this expansion after $s$ many iterations, for some fixed $s$. More precisely, one can replace the condition \eqref{bou1} with
$$
\min (n_k,m_k)\cdot\min (n_{k-1},m_{k-1})\cdots\min (n_{k-s+1},m_{k-s+1})\geq C_0^s,
$$ 
for all large $k$ and fixed $s$.
\end{remark}
Next we prove a proposition, which will allow us to rewrite the Diophantine condition \eqref{bou1} above in even more general form.

\begin{proposition}\label{main-prop}
In Theorem \ref{th1} the condition \eqref{bou1} can be replaced by
\begin{equation}\label{t1}
m_k + n_k \geq 2C_0, \hbox{ for all } k \geq k_0.
\end{equation}
\end{proposition}
\begin{proof}

As we have already mentioned, for three interval exchange maps we have $\mathcal{W}_k = \{ A_k,B_k,C_k \}$. From \eqref{f12}--\eqref{f13} it follows, that it suffices to show, that for  any $W_k \in \mathcal{W}_k$ and $W_{k-1} \in \mathcal{W}_{k-1}$, one has
\begin{equation}\label{L12}
\frac{|W_k|}{|W_{k-1}|}\geq \frac{m_k+n_k}{2}.
\end{equation}
for sufficiently large $k$.
One can check from the formulas \eqref{form1}--\eqref{form3}, that \eqref{form3} has the shortest length. Hence
\begin{equation}\label{s1}
|W_k| \geq (n_k-1) a_{k-1} + c_{k-1} + (m_{k}-1)b_{k-1}.
\end{equation}
Therefore
\begin{equation}\label{e2}
\frac{|W_k|}{|W_{k-1}|}\geq \frac{ (n_k-1)a_{k-1} + c_{k-1} + (m_{k}-1)b_{k-1}}{|W_{k-1}|}\geq \frac{ (n_k-1)a_{k-1}}{|W_{k-1}|}+ \frac{ c_{k-1}}{|W_{k-1}|} + \frac{(m_{k}-1)b_{k-1}}{|W_{k-1}|}.
\end{equation}
We have from Theorem \ref{Ft}, that $|a_k-b_k|=1$, $2a_k \geq c_k$ and $2b_k \geq c_k$. Hence
\begin{equation}\label{s2}
\frac{a_{k}}{c_{k}}\geq \frac{1}{2}, \frac{b_{k}}{c_{k}}\geq \frac{1}{2}.
\end{equation}
Similarly
\begin{equation}\label{s3}
\frac{b_{k-1}}{a_{k-1}}\geq \frac{1}{2}, \frac{b_{k-1}}{c_{k-1}}\geq \frac{1}{2}.
\end{equation}
Assume $W_{k-1}=C_{k-1}$, then from \eqref{e2} and \eqref{s3} 
\begin{alignat*}{2}
\frac{|W_k|}{|C_{k-1}|} &\geq  && \frac{ (n_k-1)a_{k-1} + c_{k-1} + (m_{k}-1)b_{k-1}}{c_{k-1}}\\
&\geq && \frac{ (n_k-1)a_{k-1}}{c_{k-1}} + 1 + \frac{(m_{k}-1)b_{k-1}}{c_{k-1}}\\
&\geq && \frac{ (n_k-1)}{2} + 1 + \frac{(m_{k}-1)}{2}=\frac{n_k + m_k}{2}.\\
\end{alignat*}
So we can assume, that $W_{k-1}\in\{A_{k-1}, B_{k-1}\}$. First, let $W_{k-1}=A_{k-1}$.  
Then
\begin{alignat*}{2}\label{h12}
\frac{|W_k|}{|A_{k-1}|} &\geq  && \frac{ (n_k-1)a_{k-1} + c_{k-1} + (m_{k}-1)b_{k-1}}{a_{k-1}}\\
&\geq && (n_k-1) + \frac{(m_{k}-1)b_{k-1}}{a_{k-1}}.\\
\end{alignat*}We want to show, that for large enough $k$
$$
(n_k-1) + \frac{(m_{k}-1)b_{k-1}}{a_{k-1}}\geq \frac{n_k + m_k}{2}.
$$
Denote $\alpha_k=b_{k-1}/a_{k-1}$. Then
$$
(n_k-1) + (m_{k}-1)\alpha_k \geq \frac{n_k + m_k}{2},
$$
or
\begin{equation}\label{L11}
\frac{n_k}{2} + m_{k}(\alpha_k - \frac{1}{2}) \geq 1 + \alpha_k.
\end{equation}
Since we have $|a_k - b_k|=1$, then clearly
\begin{equation}
\lim_{k \rightarrow \infty }\alpha_k=1.
\end{equation}
We now assume that $m_k,n_k \neq 1$. If $n_k\geq 3$, then from \eqref{L11} we will have
$$
\frac{n_k}{2} + m_{k}(\alpha_k - \frac{1}{2}) \geq \frac{3}{2} + \alpha_k - \frac{1}{2}\geq 1 + \alpha_k.
$$
If $m_k\geq 3$, then for large values of $k$
$$
\frac{n_k}{2} + m_{k}(\alpha_k - \frac{1}{2}) \geq \frac{2}{2} + 3(\alpha_k - \frac{1}{2})\geq 1 + \alpha_k.
$$
In the same way, under the assumptions $m_k,n_k \neq 1$, we can show \eqref{L12} assuming $W_{k-1}=B_{k-1}$. Now consider the case $m_k=n_k=2$. From \eqref{L12} and \eqref{e2} we need to show
\begin{equation}\label{pro:1}
\frac{a_{k-1}+c_{k-1}+b_{k-1}}{|W_{k-1}|}\geq 2.
\end{equation}

Since $|W_{k-1}|=a_{k-1}$ or $|W_{k-1}|=b_{k-1}$, $|a_{k-1}-b_{k-1}|=1$ and $c_{k-1}$ is large for large values of $k$, then
\begin{equation}\label{pro:2}
a_{k-1}+c_{k-1} \geq |W_{k-1}|, \hbox{ and } b_{k-1}+c_{k-1} \geq |W_{k-1}|,
\end{equation}
which implies \eqref{pro:1}.

We now assume, that one of the numbers $m_k$ and $n_k$ is $1$.
First let $m_k=n_k=1$. In this case, according to Proposition \ref{propos1}, $\epsilon_{k+1}$ can not be positive for infinitely many values of $k$, so we assume, that we have $(m_k,n_k,\epsilon_{k+1})=(1,1,-1)$. Hence, $W_k$ is defined by the formulas \eqref{form4}--\eqref{form6}, so the word $C_k$ has the largest length and in view of \eqref{pro:2} we will have
$$
\frac{|W_k|}{|W_{k-1}|}\geq \frac{\min\{|A_k|,|B_k|,|C_k|\}}{|W_{k-1}|}\geq\frac{\min\{c_{k-1}+b_{k-1},c_{k-1}+a_{k-1}\}}{|W_{k-1}|}\geq 1,
$$
as $W_{k-1} \in \{A_{k-1},B_{k-1}\}$.

Again from Proposition \ref{propos1}, it remains to show \eqref{L12} for $(1,m_k,-1)$ or $(n_k,1,-1)$, where $n_k,m_k \geq 2$. In the case of $(1,m_k,-1)$, from \eqref{form4}--\eqref{form6}, we have
$$
\frac{|W_k|}{|W_{k-1}|}\geq\frac{\min\{c_{k-1}+m_k b_{k-1},c_{k-1}+(m_k-1)b_{k-1}+a_{k-1}\}}{|W_{k-1}|}.
$$
Since $c_{k-1}+a_{k-1}\geq b_{k-1}$ for large $k$, then one has
$$
\frac{|W_k|}{|W_{k-1}|}\geq\frac{\min\{c_{k-1}+m_k b_{k-1},m_kb_{k-1}\}}{|W_{k-1}|}=\frac{m_kb_{k-1}}{|W_{k-1}|}>\frac{m_k+1}{2},
$$
for large values of $k$, as $b_{k-1}/|W_{k-1}|$ tends to $1$, when $k \rightarrow \infty$, and $m_k > (m_k+1)/2$, for $m_k > 1$. Similarly, for $(n_k,1,-1)$
$$
\frac{|W_k|}{|W_{k-1}|}\geq \frac{\min\{(n_k-1)a_{k-1}+c_{k-1}+b_{k-1},(n_k-1)a_{k-1}+c_{k-1}+a_{k-1}\}}{|W_{k-1}|}, 
$$
and
$$
\frac{|W_k|}{|W_{k-1}|}\geq \frac{\min\{n_ka_{k-1},n_ka_{k-1}+c_{k-1}\}}{|W_{k-1}|}=\frac{n_ka_{k-1}}{|W_{k-1}|}>\frac{n_k+1}{2},
$$
again for large values of $k$.
\end{proof}

From this proposition and in view of Remark \ref{rmk}, we arrive at the following theorem:

\begin{theorem}\label{th2}
Assume $T_{\alpha,\beta}$ is a three-interval exchange transformation satisfying the Keane condition and such that the associated three-interval expansion sequence 
$$
(n_k,m_k,\epsilon_{k+1})_{k\geq1}
$$
of integers for all $k\geq k_{\alpha,\beta}$ and for some $s\geq 1$ fulfills the conditions
\begin{equation}\label{bou}
(n_k+m_k)(n_{k-1}+m_{k-1})\cdots (n_{k-s+1}+m_{k-s+1})\geq (2C_0)^s, 
\end{equation}
where $C_0$ is as in Theorem \ref{th1}.
Then $T_{\alpha,\beta}$ satisfies Sarnak's disjointness conjecture.
\end{theorem}

We now turn to the estimation of the constant $C_0$.
In the proof of Theorem \ref{thm} Bourgain, first estimates the $L^1$ norm of the polynomials $P_W$, namely the lemmas 3 and 4 in \cite{B}. For the result it is also essential to slow down the growth of the $L^1$ norm of the polynomial $P_W$, whenever $|W_k|\rightarrow \infty$ (see Lemma 4 in \cite{B}). This condition is achieved by assuming that the lengths $|W_k|$ of the words in the symbolic representations \eqref{f11} grow sufficiently fast, i.e. conditions \eqref{f12} and \eqref{f13}. One of the key places, where this is used is Lemma 4. To estimate how big the constant $C_0$ has to be, we will follow Bourgains steps and give a more quantitative proof of this lemma.
\begin{claim}\label{clm}
The constant $C_0$ in Theorem \ref{th2} is at least required to satisfy
$$
C_0 >24^{12}.
$$
\end{claim}
\begin{proof}
First we note the following. Let $W_1 \rightarrow W_2 \cdots \rightarrow W_n$ be a sequence of words with $W_k \in \mathcal{W}_k$, where each $W_k$ participates in the symbolic representation of $W_{k+1}$. Then according to the assumptions \eqref{f12} and \eqref{f13} one has
\begin{equation}\label{q1}
\frac{|W_n|}{|W_{n-s}|}\geq \beta(s)\geq C_0^{s}.
\end{equation}
In the same way
\begin{equation}\label{q2}
\frac{|W_{n-s}|}{|W_{n-2s}|}\geq C_0^{s},\dots ,\frac{|W_{n-ks}|}{|W_{n-(k+1)s}|}\geq C_0^{s}, \dots.
\end{equation}
We note that there can only be finitely many indices, where the above inequalities do not hold, but that will not affect the estimates that follow. Multiplying together the inequalities in \eqref{q1} and \eqref{q2} we will have
$$
|W_n| \geq C' C_0^{n},
$$
which leads to 
$$
|W_n|^{\frac{1}{n}} \geq (C')^{1/n} C_0.
$$
Hence
\begin{equation}\label{c:0}
\liminf_{n \rightarrow \infty} |W_n|^{\frac{1}{n}} \geq C_0.
\end{equation}
Now back to Lemma 4 in \cite{B}. We observe that the proof of the lemma is based on the inequality $(2.13)$ in \cite{B}, i.e.
\begin{equation}\label{mainineq}
\int_0^{2\pi} |P_W(\theta)||\sum_{j=0}^k e(jl\theta)|\leq C \log(k+2)\Vert P_W \Vert_1.
\end{equation}

The proof of \eqref{mainineq}, in its turn, is based on Lemma 3. We note, that in the proof of Lemma 3 in \cite{B}, Bourgain doesn't use any particular property of the polynomial $P_W$ and since the inequality \eqref{mainineq} holds for any polynomial $P_W(\theta)$, we can assume, that $P_W(\theta)\equiv \const$. So we will end up with the following inequality
$$
\int_0^{2\pi} |\sum_{j=0}^k e(jl\theta)|\leq C \log(k+2).
$$
We know
$$
\sum _{k=0}^{N}e^{ik\theta}=e^{iN\theta/2}{\frac {\sin((N+1)\,\theta/2)}{\sin(\theta/2)}}.
$$
Hence
$$
\int_0^{2\pi} |\sum_{j=0}^k e(jl\theta)|=\int_0^{2\pi}\left|{\frac {\sin((k+1)\,l\theta/2)}{\sin(l\theta/2)}}\right|d\theta =
$$
$$
\int_0^{2\pi}\left|{\frac {\sin((k+1)\,\theta/2)}{\sin(\theta/2)}}\right|d\theta = \int_0^{2\pi}| D_k(\theta) |d\theta,
$$
where $D_k$ is the Dirichlet kernel for which one has (e.g., see \cite{Z})
$$
\| D_k \|_{L^1} \geq 4\int_0^\pi\frac{\sin t}{t}dt+\frac{8}{\pi}\log k.
$$
So we conclude, that
$$
4\int_0^\pi\frac{\sin t}{t}dt+\frac{8}{\pi}\log k \leq C \log(k+2).
$$
Dividing both sides by $\log k$ and tending $k$ to infinity we get, that
\begin{equation}
\frac{8}{\pi} \leq C.
\end{equation}
Next we estimate the $\varepsilon$ in $(2.15)$ of \cite{B}. For this we refer to the inequality $(2.26)$ in \cite{B}. We point out, that in this part of the paper, Bourgain is proving the bound \eqref{ident22}, so his goal is to estimate the following integral from above
$$
\int_\mathbb{T} |P_W(\theta)||\sum_1^N \mu(n)e(m\theta)|d\theta. 
$$
For this, in $(2.23)$ and $(2.24)$ in \cite{B}, he defines the minor and major arcs, which depend on parameters $K$ and $Q$. Lemma 6, \cite{B}, yields the inequality $(2.25)$. For arcs which are sufficiently close to rational numbers with large denominators (i.e. when $Q$ is large) or sufficiently far from the rationals with small denominator (i.e. when $K$ is large) applying Lemma 6, \cite{B}, one gets the inequality $(2.25)$ in \cite{B}. Furthermore, applying Lemma 4 and the inequality $(2.25)$ in \cite{B}, for any $Q_0\leq Q+K$ one gets
\begin{equation}
\sum_{\max(Q,K)>Q_0} \int_{V_{Q,K}} |P_W(\theta)||\sum_1^N \mu(n)e(m\theta)|d\theta \ll_{\varepsilon}(Q_0^{-\frac{1}{4}}+N^{-\tau/4})N^{1+\varepsilon}.
\end{equation}
Or dividing both sides by $N$
\begin{equation}\label{c:11}
\frac{1}{N}\sum_{\max(Q,K)>Q_0} \int_{V_{Q,K}} |P_W(\theta)||\sum_1^N \mu(n)e(m\theta)|d\theta \ll_{\varepsilon}(Q_0^{-\frac{1}{4}}+N^{-\tau/4})N^{\varepsilon}.
\end{equation}
From the above inequality it follows, that if one of the numbers $P,Q$ is sufficiently large, then the necessary estimate on the set $V_{Q,K}$ will be achieved. From here Bourgain concludes, that one can assume $Q,K< N^{\epsilon}$, for some $\epsilon$. But we see from \eqref{c:11}, that this argument will be possible only if the quantity $N^{-\tau/4}$ is small relative to $N^\varepsilon$, or
\begin{equation}
\frac{\tau}{4}>\varepsilon.
\end{equation}
But from Lemma 6 in \cite{B} we have that $0<\tau <1/3$. Hence
$$
\varepsilon < 1/12.
$$
We now return to the proof of Lemma 6 in \cite{B}. To obtain his formula $(2.15)$, in \cite{B}, Bourgain uses the inequality \eqref{mainineq} above to iterate the formulas \eqref{form1}-\eqref{form6}, i.e. for $W \in \mathcal{W}_n$ one gets
\begin{equation}\label{iter}
\| P_W \|_{L^1} \leq C\log(2+k_1)\| P_{W_1} \|_{L^1}+\cdots+C\log(2+k_r)\| P_{W_r} \|_{L^1}.
\end{equation}
Iterating further the polynomials $P_{W_1},P_{W_2},\dots,P_{W_r}$ at step $n$ we will get at most $4^n$ many members of the form
$$
C^n\log(2+k_1) \cdots \log(2+k_n)
$$
(we say $4^n$ since in the formulas \eqref{form1}--\eqref{form6} at most 4 subwords appear). Now, using the geometric arithmetic-mean inequality one gets

\begin{alignat*}{2}
\log(2+k_1) \cdots \log(2+k_n)
&\leq && \left(\frac{\log(2+k_1)+ \cdots +\log(2+k_n)}{n}\right)^n\\
&=  &&\left(\frac{\log(2+k_1)\cdots(2+k_n)}{n}\right)^n\\
&=  &&\left(\frac{\log(k_1\cdots k_n)+\log(\frac{2}{k_1}+1)\cdots(\frac{2}{k_n}+1)}{n}\right)^n\\
&\leq  &&\left(\frac{\log |W|+\log(\frac{2}{1}+1)\cdots(\frac{2}{1}+1)}{n}\right)^n\\
&\leq  &&\left(\frac{\log |W|+n\log3}{n}\right)^n\\
&=  &&\left(\log |W|^\frac{1}{n}+\log3\right)^n\\
\end{alignat*}
as
$$
\log(k_1\cdots k_n)\leq\log|W|.
$$
Therefore
\begin{equation}\label{S:1}
\| P_W \|_{L^1} \leq 4^nC^n\left(\log |W|^\frac{1}{n}+\log3\right)^n.
\end{equation}
According to $(2.24)$ in Lemma 4 in \cite{B} the above expression has to be smaller then $|W|^\epsilon$, i.e.
\begin{equation}\label{how}
4^nC^n\left(\log |W|^\frac{1}{n}+\log3\right)^n \leq |W|^\epsilon.
\end{equation}
However we will neglect the coefficient $4^n$ in the above inequality (which amount to saying, that at each step of the iteration of \eqref{iter} we have only $1$ word, or $W_{k+1}$ is a power of $W_k$). In other words, instead of \eqref{how} we will consider the following inequality 
$$
C^n\left(\log |W|^\frac{1}{n}+\log3\right)^n \leq |W|^\epsilon,
$$
which is clearly implied by \eqref{how}.
Since $\varepsilon<1/12$, then the above inequality will also imply
$$
\left(C\log |W|^\frac{1}{n}+C\log3\right)^n < |W|^{1/12},
$$
or
$$
C\log |W|^\frac{1}{n}+C\log3 < |W|^{\frac{1}{12n}}.
$$
Denote $x=|W|^\frac{1}{n}$. Hence
$$
C\log x+C\log3 < x^{1/12}.
$$
Therefore, if $(2.24)$ in \cite{B} holds, then so does the inequality above.
Now consider
$$
f(x)= x^{1/12} - C\log x -C\log3,
$$
and compute
$$
\frac{d}{dx}f(x)=\frac{1}{12}x^{\frac{1}{12}-1}-\frac{C}{x}=0.
$$
For the critical point of $x_c$ we have
$$
x_c=(12C)^{12}>(12\frac{8}{\pi})^{12}>24^{12}.
$$
But
$$
f(x_c)=12C-12C\log(12C)-C\log3<0,
$$
so we see, that if $f(x)>0$ for some large $x$, then we must have $x>x_c=24^{12}$. But we know from \eqref{c:0}, that $x=|W|^\frac{1}{n}>C_0$ for large $n$. This finishes the proof of Claim \ref{clm}.
\end{proof}

As we see the constant $C_0 $ in Theorem \ref{th2} must be very large. But in the present paper we will only assume that $C_0\geq 20$. 

\medskip
We are now ready to formulate the main theorem of this paper:
\begin{theorem}[Main theorem]\label{mainth}
Under the conditions of Theorem \ref{th1}, Sarnak's disjointness conjecture holds for all three-interval exchange maps $T_{\alpha,\beta}$, $(\alpha,\beta)\in D_0$, for which their associated three-interval expansion sequence $(n_k,m_k,\epsilon_{k+1})_{k\geq1}$ fulfills the conditions
\begin{equation}\label{bou-mth}
(n_k+m_k)(n_{k-1}+m_{k-1})\cdots (n_{k-s+1}+m_{k-s+1})\geq (2C_0)^s, 
\end{equation}
for all $k\geq k_{\alpha,\beta}$ and some $s\geq 1$
and for the Hausdorff dimension of the set
\begin{equation}
 \mathcal{P}_0=\{(\alpha,\beta )\in D_0: \hbox{which satisfy \eqref{bou-mth}}, \hbox{ for all } k \geq k_{\alpha,\beta} \hbox{ and some } s \geq 1 \},
\end{equation}
for $C_0\geq 20$ we have the following estimates
\begin{equation}
\frac{3}{2} + \frac{1}{2\log(2C_0+2)} \leq \dim_H \mathcal{P}_0\leq 1+t(\log \Lambda), 
\end{equation}
where $$\Lambda=({2C_0/3})^{1/2},$$ and the function $t=t(\zeta)$ is defined in Theorem \ref{t-func} (Figure 1), see \cite{Fan}. In particular
\begin{equation}
\frac{3}{2}<\dim_H \mathcal{P}_0<2. 
\end{equation}

\end{theorem}

\begin{corollary}\label{cor1}
If $C_0\geq 20$, then the two dimensional Lebesgue measure of the set $\mathcal{P}_0$ is zero.
\end{corollary}

\section{Estimates on Hausdorff dimension}\label{sec:2}

We first recall the definition of Hausdorff dimension. Let $X$ be a metric space. If $S \subset X$ and $d \in [0, \infty)$, the $d$-dimensional Hausdorff content of $S$ is defined by
$$
{\displaystyle C_{H}^{d}(S):=\inf {\Bigl \{}\sum _{i}r_{i}^{d}:{\text{ there is a cover of }}S{\text{ by balls with radii }}r_{i}>0{\Bigr \}}.} 
$$
In other words, ${\displaystyle C_{H}^{d}(S)}$ is the infimum of the set of numbers $\delta > 0$ such that there is some (indexed) collection of balls $\{B(x_i,r_i):i\in I\}$ covering $S$ with $r_i > 0$ for each $i \in I$ that satisfies $\sum_{i\in I} r_i^d<\delta$ . 
Then the Hausdorff dimension of X is defined by
$$
\dim _{{\operatorname {H}}}(S):=\inf\{d\geq 0:C_{H}^{d}(S)=0\}.
$$
We will need the following classical facts about this concept, (see, e.g. \cite{Sh}, Theorem 2).
\begin{theorem}\label{h11}
If $f : X \rightarrow f (X)$ is a Lipschitz map, then $\dim_H( f (X)) \leq \dim_H (X)$.
\end{theorem}
\begin{theorem}\label{thm2}
If $X_i$ is a countable collection of sets with $\dim_H$ $(X_i) \leq d$, 
then $\dim_H (\cup_i X_i)=\sup_{i}\dim_HX_i\leq d $.
\end{theorem}

Next we prove the following proposition.
\begin{proposition}\label{p1}
Let $H \in \mathcal{H}$, then $H^{-1}$ maps $D_0$ into $D_0$ and it is a Lipschitz map.
\end{proposition}
\begin{proof}
It is enough to show this for the maps $F$ and $G$. By definition $D_0$ is the region bounded by the lines $y = 0$, $x = 0$, and $x + y = 1$. The inverse of $F$ and $G$ can be computed as
$$
F^{-1}(x_1,y_2)=\left( \frac{1}{2-x_1}, \frac{y_1}{2-x_1} \right),
$$
and
$$
G^{-1}(x_1,y_2)=\left( 1 - x_1 - y_2 , y_2 \right).
$$
If $x_1 + y_1 \leq 1$, then considering the two coordinates of $F^{-1}$ we have
$$
\frac{1}{2-x_1} + \frac{y_1}{2-x_1} \leq \frac{1 + 1 - x_1}{2-x_1}=1
$$
as $y_1 \leq 1 - x_1$. Hence
$$
\left( \frac{1}{2-x_1}, \frac{y_1}{2-x_1} \right) \in D_0.
$$
Similarly for $G^{-1}$, if $x_1 + y_1 \leq 1$
$$
1 - x_1 - y_1 + y_1 = 1 - x_1 \leq 1,
$$
hence
$$
\left( 1 - x_1 - y_1 , y_1 \right) \in D_0.
$$
To prove that they are Lipschitz it is sufficiently to show, that the partial derivatives are uniformly bounded in $D_0$.
$$
\frac{\partial}{\partial x_1}\left(\frac{1}{2-x_1}\right)=\frac{1}{(2-x_1)^2}\leq 1,
$$
$$
\frac{\partial}{\partial y_1}\left(\frac{1}{2-x_1}\right)=0,
$$
$$
\frac{\partial}{\partial y_1}\left(\frac{y_1}{2-x_1}  \right)= \frac{1}{2-x_1}\leq 1,
$$
$$
\frac{\partial}{\partial x_1}\left(\frac{y_1}{2-x_1}  \right)= \frac{y_1}{(2-x_1)^2}\leq 1,
$$
as $x_1,y_1 \leq 1$. Since for any $H \in \mathcal{H}$, $H^{-1}$ is a composition of Lipschitz functions, i.e.
$$
H^{-1}(x,y)= G^{-t} \circ F^{-l_0} \circ G^{-1} \circ F^{-l_1} \circ G^{-1} ... \circ G^{-1} \circ F^{-l_k} \circ G^{-s}(x,y)
$$
then $H^{-1}$ is also Lipschitz.
\end{proof}
Recall the following definition from Theorem \ref{mainth}
$$
 \mathcal{P}_0=\{(\alpha,\beta )\in D_0: \hbox{which satisfy }\eqref{bou-mth}, \hbox{ for all } k \geq k_{\alpha,\beta} \hbox{ and some } s \geq 1 \}.
$$
Define also
\begin{equation}
\mathcal{P}=\{(\alpha,\beta )\in D:  \hbox{which satisfy }\eqref{bou-mth}, \hbox{ for all } k \geq k_{\alpha,\beta} \hbox{ and some } s \geq 1 \}.
\end{equation}
From the discussion at the beginning of Section 2 and the definition of the sequence $\{n_k, m_k,\epsilon_{k+1}\}_{k=1}^\infty$ we have, that
\begin{equation}\label{c11}
\mathcal{P}_0 \subset \bigcup_{H \in \mathcal{H}} H^{-1}(\mathcal{P}).
\end{equation}

\begin{corollary}\label{maincor}
Assume $\dim_H{\mathcal{P}}\leq d$. Then
$$
\dim_H\mathcal{P}_0\leq d.
$$
\end{corollary}
\begin{proof}
According to Proposition \ref{p1}, for any $H \in \mathcal{H}$, $H^{-1}$ is Lipschitz. Therefore from Theorem \ref{h11} it follows, that
$$
\dim_H H^{-1}(\mathcal{P}) \leq \dim_H(\mathcal{P}).
$$
From this, Theorem \ref{thm2} and \eqref{c11} we will have
$$
\dim_H\mathcal{P}_0\leq\dim_H \left(\bigcup_{H \in \mathcal{H}} H^{-1}(\mathcal{P}) \right)\leq d.
$$
\end{proof}
Therefore, to estimate the Hausdorff dimension of $\mathcal{P}_0$, it is enough to estimate it for $\mathcal{P}$, i.e. when $(\alpha,\beta)\in D$ and for this $(\alpha,\beta)$'s one has the following relation
$$
\frac{1 - \alpha}{1 + \beta} = \cfrac{1}{2 + \cfrac{1}{m_1  + n_1 - \cfrac{\epsilon_2}{m_2 + n_2 - \cfrac{\epsilon_3}{m_3 + n_3 - \ddots\,}}}}.
$$

\bigskip
We now recall the definition of standard continued fractions. For any $\theta\in [0,1]$ its continued fraction is an expression of the form
$$
\theta = a_0 + \cfrac{1}{a_1 + \cfrac{1}{a_2 + \cfrac{1}{a_3 + \ddots }}},
$$
and its $n$'th convergent is denoted by
$$
\frac{p_n}{q_n}=\cfrac{1}{a_1 + \cfrac{1}{a_2  + \ddots\, + \cfrac{1}{a_n }}}.
$$
With the conventions $p_{-1} = 1$, $q_{-1} =0$, $p_0 = 0$, $q_0 = 1$, we have
$$
p_{n+1} = a_{n+1} p_n + p_{n-1}, q_{n+1} = a_{n+1} q_n + q_{n-1}, (n \geq 0),
$$
and
\begin{equation}\label{c1}
p_{n+1}q_{n}-p_nq_{n+1}=(-1)^n,\quad (n \geq -1).
\end{equation}

Let $\mathcal{I}^n = \mathcal{I}^n_{a_1,a_2,..,a_n}$, where $a_1,a_2,..,a_n$ are positive integers, be the interval
$$
\left(\cfrac{1}{a_1 + \cfrac{1}{a_2  + \ddots\, + \cfrac{1}{a_n }}} , \cfrac{1}{a_1 + \cfrac{1}{a_2  + \ddots\, + \cfrac{1}{a_n + 1}}} \right),
$$
that is
\begin{equation}
\left(\frac{p_n}{q_n}, \frac{p_n + p_{n+1}}{q_n + q_{n+1}} \right).
\end{equation}
Here, for $a \neq b$, we mean by $(a, b)$ the closed interval with end-points $a, b$. This means, that we can also have $b<a$.
From \eqref{c1} it follows that the length $|\mathcal{I}^n|$ of $\mathcal{I}^n$ satisfies
\begin{equation}\label{c2}
|\mathcal{I}^n|=\frac{1}{q_n(q_n + q_{n-1})}.
\end{equation}
We will also work with more general kind of continued fractions, namely semi-regular continued fractions (SRCF). For $\theta \in [0,1]$ its SRCF expansion looks like this
\begin{equation}\label{f1}
\theta = \cfrac{1}{a_1 + \cfrac{\epsilon_1}{a_2 + \cfrac{\epsilon_2}{a_3 + \cfrac{\epsilon_3}{a_3 + \ddots\,}}}},
\end{equation}
where $\epsilon_k=\pm1$ and $a_k \geq 2$ for all $k \geq 1$.
For short we will write \eqref{f1} in the following way
$$
\theta=[\epsilon_1/a_1, \epsilon_2/a_2,\dots,\epsilon_n/a_n,\dots].
$$  
Note, that if $\epsilon_k = 1$ for all $k \geq 1$, then we get the standard continued fraction expansion of $\theta$. The SRCF expansion is defined for $a_k\geq 2$, but we will also deal with the cases, when $a_k=1$ or $a_k=0$.

The following identity will be fundamental for us (see e.g. \cite{Kra-Iosif}). For $a \in \mathbb{Z}$,  $b \in N_+$ and $x \in [0,1)$ we have 
\begin{equation}\label{id}
a + \cfrac{-1}{b + x}=a - 1 + \cfrac{1}{1 + \cfrac{1}{b - 1  + x }}.
\end{equation}
Using this identity we are going to find the standard continued fraction expansion from their SRCF expansion.
According to Proposition 1, if $(\alpha,\beta)\in D$, then
$$
A({\alpha},{\beta}) = \frac{1 - \alpha}{1 + \beta} = \cfrac{1}{2 + \cfrac{1}{m_1  + n_1 - \cfrac{\epsilon_2}{m_2 + n_2 - \cfrac{\epsilon_3}{m_3 + n_3 - \ddots\,}}}}.
$$
In other words
\begin{equation}\label{b:1}
\frac{1 - \alpha}{1 + \beta}= [1/2,-\epsilon_2/(m_1 + n_1),\dots, -\epsilon_{k+1}/(m_k + n_k), \dots ]. 
\end{equation}
To not carry the minus sign in \eqref{b:1} all the time in the computations we will simply replace the $-\epsilon_k$ with $\epsilon_k$. So from now on \eqref{b:1} will look like this
\begin{equation}\label{b:2}
\frac{1 - \alpha}{1 + \beta}= [1/2,\epsilon_2/(m_1 + n_1),\dots, \epsilon_{k+1}/(m_k + n_k), \dots ],
\end{equation}
where $\epsilon_k=\pm 1$.
Now we will use the identity \eqref{id} to get rid of the negative $\epsilon_k$'s. If in \eqref{b:2} $\epsilon_{k+1}=-1$ for some $k$, then from \eqref{id}
\begin{equation}\label{b:4}
m_k + n_k + \cfrac{-1}{m_{k + 1} + n_{k + 1} + \cfrac{\epsilon_{k+2}}{x}} = m_k + n_k - 1 + \cfrac{1}{1 + \cfrac{1}{m_{k + 1} + n_{k + 1} -1  + \cfrac{\epsilon_{k+2}}{x} }},
\end{equation}
or
\begin{equation}\label{f2}
\begin{aligned}
 &\frac{1 - \alpha}{1 + \beta}=[\dots, -1/(m_{k} + n_{k}), \epsilon_{k+2}/(m_{k+1} + n_{k+1}), \dots ] \\
 &=[\dots , 1/(m_{k} + n_{k}-1),1/1 ,\epsilon_{k+2}/(m_{k+1} + n_{k+1}-1), \dots ]. \\
\end{aligned}
\end{equation}
By definition $m_k,n_k \geq 1$ and hence $m_k + n_k \geq 2$.
As we see from the equation above, any number $(n_k+m_k)$ can participate in at most two replacements, hence can be reduced by at most $2$, i.e. become $n_k+m_k-2$. This will be the case with $m_{k + 1} + n_{k + 1} -1$ in \eqref{b:4} if $\epsilon_{k+2}=-1$. If $m_{k + 1} + n_{k + 1} - 2>0$, then the replacement will be valid. The case $m_{k + 1} + n_{k + 1}= 2$ needs special considerations. From the left hand side in \eqref{b:4}

\begin{equation}\label{b:3}
m_{k + 1} + n_{k + 1} - 1 + \cfrac{-1}{m_{k + 2} + n_{k + 2} + \cfrac{\epsilon_{k+3}}{x_1}} = m_{k + 1} + n_{k + 1} - 2 + \cfrac{1}{1 + \cfrac{1}{m_{k + 2} + n_{k + 2} -1  + \cfrac{\epsilon_{k+3}}{x_1} }},
\end{equation}
and since $m_{k+1} + n_{k+1} = 2$, then
$$
m_{k + 1} + n_{k + 1} - 1 + \cfrac{-1}{m_{k + 2} + n_{k + 2} + \cfrac{\epsilon_{k+3}}{x_1}} = 
\cfrac{1}{1 + \cfrac{1}{m_{k + 2} + n_{k + 2} -1  + \cfrac{\epsilon_{k+3}}{x_1}}}.
$$
Putting this back into \eqref{b:4} we get
\begin{alignat*}{2}\label{h12}
m_k + n_k + \cfrac{-1}{m_{k + 1} + n_{k + 1} + \cfrac{\epsilon_{k+2}}{x}}  & =  && m_k + n_k - 1 + \cfrac{1}{1 + \cfrac{1}{
\cfrac{1}{1 + \cfrac{1}{m_{k + 2} + n_{k + 2} -1  + \cfrac{\epsilon_{k+3}}{x_1}}}}}\\
& = && m_k + n_k - 1 + \cfrac{1}{2 + \cfrac{1}{m_{k + 2} + n_{k + 2} -1  + \cfrac{\epsilon_{k+3}}{x_1}}}.\\
\end{alignat*}
In a similar way, if we have $m_{k + s} + n_{k + s}=2$, $\epsilon_{k + l + 1}=-1$ for all $s=1,\dots,l$, and either $m_{k + l+1} + n_{k + l+1}>2$ or $\epsilon_{k + s + 2}=1$ and then one can show, that the equality above, can be rewritten as follows
\begin{equation}\label{f:2}
\begin{aligned}
 &[\dots,-1/(m_{k} + n_{k}),-1/2,\dots,-1/2 ,{\epsilon_{k+l+2}}/(m_{k+l+1} + n_{k+l+1}), \dots ] \\
 &=[\dots,1/(m_{k} + n_{k}-1),1/(l+1),{\epsilon_{k+l+2}}/(m_{k+l+1} + n_{k+l+1}-1), \dots ]. \\
\end{aligned}
\end{equation}
Observe, that according to \eqref{bou-mth}, we should have $l < k_{\alpha,\beta}$.
We see that as a result of this procedure we will get the continued fraction expansion of $(1 - \alpha)/(1 + \beta)$. 

The next proposition shows, that we have a \eqref{bou} like property also for the standard continued fractions of $(1 - \alpha)/(1 + \beta)$:

\begin{proposition}\label{pro}
Let $(\alpha,\beta)\in \mathcal{P}$ and 
\begin{equation}
\frac{1 - \alpha}{1 + \beta}= [1/2,-\epsilon_2/(m_1 + n_1),\dots, -\epsilon_{k+1}/(m_k + n_k), \dots ],
\end{equation}
then for the standard continued fractions of $(1 - \alpha)/(1 + \beta)$, i.e. $(1 - \alpha)/(1 + \beta)=[a_1,a_2,\dots]$ there is a number $C_{\alpha,\beta}\in \mathbb{N}$ so that for any large $n$ there is a number $s_{\alpha,\beta} \in \mathbb{N}$, with $s_{\alpha,\beta}\leq C_{\alpha,\beta}$, such that
$$
(a_na_{n-1}\cdots a_{n-s_{\alpha,\beta}})^\frac{1}{s_{\alpha,\beta}}\geq \Lambda,
$$
where $\Lambda=({2C_0/3})^{1/2}$.
\end{proposition}
\begin{proof}
We know from \eqref{bou}, that
$$
(n_k+m_k)(n_{k-1}+m_{k-1})\cdots (n_{k-s+1}+m_{k-s+1})\geq (2C_0)^s.
$$
Let the number of $2$'s and $3$'s between the numbers $\{(m_j+n_j)\}_{j=k-s+1}^{k}$ be respectively equal to $m_1$ and $m_2$.
Therefore
$$
(n_k+m_k)(n_{k-1}+m_{k-1})\cdots (n_{k-s+1}+m_{k-s+1})= (n_{k_1} + m_{k_1})\cdots (n_{k_l} + m_{k_l}) 2^{m_1}3^{m_2},
$$
where
$$
l+m_1+m_2=s,
$$
and
$$
(n_{k_j}+m_{k_j})\geq 4,
$$
for $j=k_1,\dots, k_s$.
Therefore
$$
(n_{k_1} + m_{k_1})\cdots (n_{k_l} + m_{k_l}) \geq \frac{(2C_0)^s}{2^{m_1}3^{m_2}}.
$$
Assume, that the digits
\begin{equation}\label{pp}
\{(n_k+m_k),(n_{k-1}+m_{k-1}),\dots ,(n_{k-s+1}+m_{k-s+1})\}, 
\end{equation}
during the procedure described above have transformed into new $r$ many digits, i.e corresponding to the standard continued fractions
\begin{equation}\label{res}
\{a_i,a_{i-1},\dots ,a_{i-r+1}\}.
\end{equation}
In case we have digit 1's appearing on both sides of the continued fraction $(n_j+m_j)$, then in \eqref{res} we will include only the left digit 1.

In case $\epsilon_j=-1$ for all $j=k-s+1, \dots, k$, we will have a new digit $1$ appearing in between any two digits $(n_j+m_j)$ and $(n_{j-1}+m_{j-1})$. Therefore the number of digits in \eqref{pp} will at most double, i.e.
\begin{equation}\label{est:m}
r \leq 2s.
\end{equation}
And since each digit may participate in at most two replacements, then clearly
$$
a_ia_{i-1}\cdots a_{i-r}\geq (n_{k_1} + m_{k_1}-2)\cdots (n_{k_l} + m_{k_l}-2).
$$
Now note, that if $m_{k_j} + n_{k_j}\geq 4$, then
$$
m_{k_j} + n_{k_j}-2 \geq \frac{m_{k_j} + n_{k_j}}{2}.
$$ 

Therefore
\begin{alignat*}{2}\label{h12}
a_ia_{i-1}\cdots a_{i-r+1}  & \geq  && \frac{(m_{k_1}+n_{k_1})}{2} \cdots \frac{(m_{k_l}+n_{k_l})}{2}\\
& \geq && \frac{(2C_0)^s}{2^{l}\cdot 2^{m_1}\cdot 3^{m_2}}=\frac{(2C_0)^s}{2^{s-m_1-m_2}\cdot 2^{m_1}\cdot 3^{m_2}}\\
& = && \frac{(2C_0)^s}{2^{s-m_2}\cdot 3^{m_2}}\geq \frac{(2C_0)^s}{3^{s}}\\
& = && \left(\frac{2C_0}{3}\right)^s.\\
\end{alignat*}
Thus
$$
a_ia_{i-1}\cdots a_{i-r+1} \geq \left(\frac{2C_0}{3}\right)^s,
$$
or
$$
(a_ia_{i-1}\cdots a_{i-r+1})^{1/r} \geq \left(\frac{2C_0}{3}\right)^{s/r}.
$$
And since $r\leq 2s$, then 
$$
\frac{1}{2}\leq \frac{s}{r}.
$$
Therefore
$$
(a_ia_{i-1}\cdots a_{i-r+1})^{1/r} \geq \left(\frac{2C_0}{3}\right)^{1/2}.
$$
We can also see from \eqref{est:m} that $r$ is uniformly bounded, since
$$
r \leq 2s \leq 2k_{\alpha,\beta}.
$$
\end{proof}
By definition
$$
\Lambda = \left(\frac{2C_0}{3}\right)^{1/2},
$$
and since by assumption $C_0\geq 20$, then
\begin{equation}\label{estm:1}
\Lambda \geq 3.
\end{equation}
We recall, that our goal is to estimate the Hausdorff dimension of the set $\mathcal{P}$. Define
\begin{equation}\label{s:0}
S_0=A(\mathcal{P})=\{\theta \in [0,1]: \theta = A(\alpha,\beta), \hbox{ where } (\alpha,\beta)\in \mathcal{P}\},
\end{equation}
where the function $A$ is defined in \eqref{func}.
Notice, that the set
\begin{equation}\label{S}
\begin{aligned}
S & = \{\theta \in [0,1]:\exists k_\theta, s_\theta \geq 1, \hbox{ such that, if } k \geq k_\theta, \hbox{ then } (a_na_{n-1}\cdots a_{n+1-s_\theta})^{\frac{1}{s_\theta}}\geq \Lambda \}, \\
\end{aligned}
\end{equation}
from Proposition \ref{pro}, satisfies the inclusion
\begin{equation}\label{imp}
S_0 \subset S.
\end{equation}
It is not difficult to see, that for any $\theta \in S$
\begin{equation}\label{Kh:1}
\limsup_{n \rightarrow \infty}(a_1a_2\cdots a_n)^{\frac{1}{n}}\geq  \Lambda.
\end{equation}
Alternatively
\begin{equation}\label{Kh:2}
\limsup_{n \rightarrow \infty}\frac{1}{n}\sum_{k=1}^n \log a_k \geq  \log \Lambda.
\end{equation}
 One can now see from Khinchine's theorem \cite{Kh}, that for sufficiently large $\Lambda$ the Lebesgue measure of the set $S$ is zero. Indeed, according to Khintchine's theorem
$$
\lim_{n \rightarrow \infty} (a_1a_2\cdots a_n)^{\frac{1}{n}}=K_0, \hbox{ for almost all } \theta \in [0,1],
$$ 
where
$$
K_0\approx 2.68545.
$$
Clearly, in view of \eqref{estm:1} for any $\theta \in S$
\begin{equation}\label{khin}
\limsup_{n \rightarrow \infty} (a_1a_2\cdots a_n)^{\frac{1}{n}}\geq \Lambda\geq 3 > K_0.
\end{equation}
Therefore the Lebesgue measure of the set $S$ is zero. From this one can see, that the two dimensional Lebesgue measure of $(\alpha,\beta)\in D_0$, for which $A(\alpha,\beta)\in S$, is also zero. However this will also follow from Corollary \ref{cor1}.

In \cite{Fan} the authors, alongside with other things, for each $\gamma>0$, compute the Hausdorff dimension of the set of all $x \in [0,1]$, for which the following limits exists and equals
$$
\lim_{n \rightarrow \infty}\frac{1}{n}\sum_{j=1}^n\log a_j(x)=\gamma,
$$
or equivalently, if
$$
\lim_{n \rightarrow \infty}(a_1(x)a_2(x)\cdots a_n(x))^{\frac{1}{n}}=e^\gamma.
$$
In this paper we need to estimate the Hausdorff dimension of the set $S_0$, where the continued fractions, in particular, satisfy the property \eqref{khin}. Therefore we need, in a sense, stronger result. We will show in the sequel, that the method used in \cite{Fan} will allow to achieve this.

\medskip

\begin{proposition}\label{prop:3}
For $C_0\geq 20$, the Hausdorff dimension of the set $S_0$ satisfies the following bounds
\begin{equation}\label{fin1}
\frac{1}{2} + \frac{1}{2\log(2C_0+2)}\leq \dim_H S_0 \leq t(\log\Lambda),
\end{equation}
where the function $t(\zeta)$ is defined in Theorem \ref{t-func}.
\end{proposition}
\begin{proof}
We recall certain facts from \cite{Fan}.
Let
$$
D:=\{(t,q)\in \mathbb{R}:2t-q>1 \},
$$
$$
D_0:=\{(t,q)\in \mathbb{R}:2t-q>1, 0\leq t \leq 1 \}.
$$
For $(t,q)\in D$, define
$$
P(t,q)=\lim_{n \rightarrow \infty} \frac{1}{n}\log \sum_{\omega_1=1}^{\infty}\cdots\sum_{\omega_n=1}^{\infty}\exp\left( \sup_{x \in [0,1]}\log \prod_{j=1}^n \omega_j^q(\omega_j,\dots,\omega_n+x])^{2t} \right).
$$
It is shown in \cite{Fan}, that $P(t,q)$ is an analytic function in $D$. Moreover, for any $\zeta\geq 0$, there exists a unique solution $(t(\zeta),q(\zeta))\in D_0$ to the equation
 
$$
\begin{cases} 
P(t,q)=q\zeta, \\
\frac{\partial P}{\partial q}(t,q)=\zeta.\\
\end{cases}
$$

In \cite{Fan} the authors study the Khintchine exponents and the Lyapunov exponents, that is for $x\in [0,1]$ the numbers
$$
\gamma(x):=\lim_{n \rightarrow \infty}\frac{1}{n}\sum_{j=1}^n\log a_j(x)=\lim_{n \rightarrow \infty}\frac{1}{n}\sum_{j=0}^{n-1}\log a_1(T^j(x)),
$$
$$
\lambda(x):=\lim_{n \rightarrow \infty}\frac{1}{n}\log|(T^n)'(x)|=\lim_{n \rightarrow \infty}\frac{1}{n}\sum_{j=0}^{n-1}\log|T'(T^j(x))|,
$$
if they exist. In the above formulas $T$ is the Gauss map, i.e.
$$
T(x)=\frac{1}{x}-\left[\frac{1}{x}\right],
$$
which is known to preserve the measure
$$
d\mu_G=\frac{dx}{(1+x)\log2}.
$$
From Birkhoff's ergodic theorem we have, that
\begin{equation}\label{cons}
\zeta_0=\log K_0=\int_0^1 \log a_1(x)d\mu_G.
\end{equation}

For real numbers $\zeta,\beta \geq 0$ one considers the level sets of Khintchine
exponents and Lyapunov exponents
$$
E_\zeta=\{x\in [0,1]: \gamma(x)=\zeta \},
$$
$$
F_\zeta=\{x\in [0,1]: \lambda(x)=\beta \}.
$$
The following theorem holds.

\begin{theorem}[\cite{Fan}]\label{t-func}
Let $\zeta_0$ be as in \eqref{cons}. For $\zeta \geq 0$, the set $E_\zeta$ is of Hausdorff
dimension $t(\zeta)$. Furthermore, the dimension function $t(\zeta)$ has the following properties:

$1)$
 $t(\zeta_0) = 1$ and $t(+\infty) = \frac{1}{2}$;
 
$2)$ $t'(\zeta)<0$ for all $\zeta>\zeta_0$, $t'(\zeta_0)=0$, and $t'(\zeta)>0$ for all $\zeta<\zeta_0$.

$3)$ $t'(0+)=\infty$ and $t'(+\infty)=0$

$4)$ $t''(\zeta_0)<0$ and $t''(\zeta_1)>0$ for some $\zeta_1>\zeta_0$, so $t(\zeta)$ is neither convex nor concave.

\end{theorem}

\begin{center}\label{fig}
\includegraphics[scale=0.5]{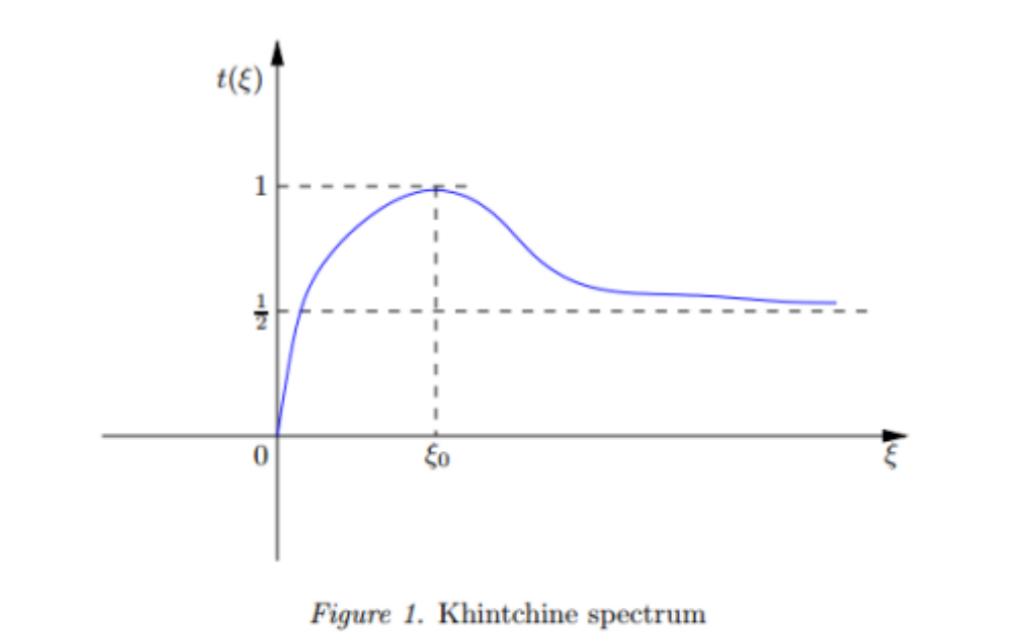}
\end{center}

Next we recall certain fact from pp. 100--101 in \cite{Fan}.

For any $\tau>t(\zeta)$ choose an $\epsilon=\epsilon(\tau)>0$ so, that
\begin{equation}\label{m:1}
0 < \epsilon < \frac{P(t(\zeta),q(\zeta)))-P(\tau,q(\zeta))}{q(\zeta)}, \hbox{ if } q(\zeta)>0,
\end{equation}
and
\begin{equation}\label{m:2}
0 < \epsilon < \frac{P(\tau,q(\zeta)))-P(t(\zeta),q(\zeta))}{q(\zeta)}, \hbox{ if } q(\zeta)<0.
\end{equation}
Such an $\epsilon$ exists, since $P(\tau,q)$ is strictly decreasing with respect to $\tau$, see pp. 100-101 in \cite{Fan}. Let $\mathcal{I}(n,\zeta,\epsilon)$ be the collection of all $n$-th order cylinders $\mathcal{I}_n(a_1,\dots, a_n)$, such that
$$
\zeta - \epsilon < \frac{1}{n}\sum_{j=1}^{n}\log a_j(x)<\zeta + \epsilon.
$$
Let
$$
E_{\zeta}^n(\epsilon)=\bigcup_{J \in \mathcal{I}(n,\zeta,\epsilon)}J.
$$
It is shown in \cite{Fan}, page 101, that
\begin{equation}\label{m:eq}
\sum_{n=1}^\infty \sum_{J \in \mathcal{I}(n,\zeta,\epsilon)}|J|^\tau < \infty.
\end{equation}
Now consider the set
\begin{equation}
A(\zeta,\epsilon)=\{x \in [0,1]: \frac{1}{n}\sum_{k=1}^{n} \log a_k(x) \in (\zeta-\epsilon, \zeta+\epsilon), \hbox{ for infinitely many }n \},
\end{equation}
or alternatively
$$
A(\zeta,\epsilon)= \bigcap_{n=1}^\infty \bigcup_{k=n}^{\infty}E_{\zeta}^k(\epsilon).
$$
From \eqref{m:eq} it follows, that
\begin{equation}\label{main:1}
\dim_H (A(\zeta,\epsilon))\leq\tau.
\end{equation}
One can also see, that if $\epsilon'<\epsilon$, then
\begin{equation}\label{mon}
\dim_H (A(\zeta,\epsilon'))<\dim_H (A(\zeta,\epsilon))\leq\tau.
\end{equation}
Now let $\zeta,\zeta_1$ be such, that
$$
\zeta_1 < \zeta < \infty
$$
and
$$
\log\Lambda>\zeta_1>\zeta_0.
$$
The choice of the number $\log\Lambda$ comes from \eqref{Kh:2}.
From the monotonicity of $t(\zeta)$ (see Theorem \ref{t-func} and Figure 1), for $\zeta > \zeta_0$ we will have
$$
t(\zeta) <t(\zeta_1).
$$
Therefore, for $\tau=t(\zeta_1)$, in view of \eqref{mon} one can choose $\epsilon$ in such a way, that
\begin{equation}\label{cov:1}
(\zeta - \epsilon, \zeta + \epsilon)\subset (\zeta_1,\infty),
\end{equation}
and \eqref{main:1} holds for $\tau$. Let
$$
\{ (\zeta-\epsilon_{\zeta}, \zeta+\epsilon_{\zeta}) \}_{\zeta > \zeta_1},
$$
be the collection of all these intervals, for $\zeta > \zeta_1$. Clearly
\begin{equation}\label{fum}
[\log\sqrt{\Lambda},\infty) \subset \bigcup_{\zeta > \zeta_1 }^\infty (\zeta-\epsilon_{\zeta}, \zeta+\epsilon_{\zeta}).
\end{equation}
By representing the closed halfinterval $[\log\Lambda,\infty)$ as a union of countably many closed intervals and using the Heine--Borel lemma, we can find a countable sub-family of intervals from \eqref{fum}
$$
\{ (\zeta_k-\epsilon_{\zeta_k}, \zeta_k+\epsilon_{\zeta_k}) \}_{k \geq 1},
$$
so that
\begin{equation}
[\log\Lambda,\infty) \subset \bigcup_{k=1}^\infty (\zeta_k-\epsilon_{\zeta_k}, \zeta_k+\epsilon_{\zeta_k}).
\end{equation}
Hence, in view of \eqref{main:1}, \eqref{cov:1}, we will have
\begin{equation}\label{s:5}
\dim_H \left(\bigcup_{k=1}^\infty A(\zeta_k, \epsilon_{\zeta_k}) \right) \leq \sup_{k \geq 1}\dim_H A(\zeta_k, \epsilon_{\zeta_k}) \leq \tau=t(\zeta_1).
\end{equation}
Consider now the following two sets $A_1$ and $A_2$:
$$
A_1=\{x:\lim_{n \rightarrow \infty} \frac{1}{n}\sum_{k=1}^n \log a_k(x) = \infty\},
$$
and
\begin{equation}
A_2=\{x: \exists n_s(x)\nearrow \infty,\lim_{n_s \rightarrow \infty} \frac{1}{n_s}\sum_{k=1}^{n_s} \log a_k(x) = \zeta_x, \hbox{ for some }\zeta_x \geq \log\Lambda \}.
\end{equation}
In follows from \eqref{Kh:2}, that
$$
S \subset A_1 \cup A_2.
$$
Since
$$
\zeta_x \in [\log\Lambda,\infty) \subset \bigcup_{k=1}^\infty (\zeta_k-\epsilon_{\zeta_k}, \zeta_k+\epsilon_{\zeta_k}),
$$
then
$$
A_2 \subset \bigcup_{k=1}^\infty A(\zeta_k, \epsilon_{\zeta_k}).
$$
Therefore, from \eqref{s:5}
\begin{equation}
\dim_H A_2 \leq t(\zeta_1).
\end{equation}
As for the set $A_1$, one has
$$
\dim_HA_1=t(\infty)=\frac{1}{2}<t(\zeta_1).
$$
Therefore
\begin{equation}
\dim_H S \leq \dim_H \left( A_1 \cup A_2 \right)\leq t(\zeta_1).
\end{equation}
But since $\zeta_1$ was an arbitrary number between $\log \Lambda$ and $\zeta_0$, and $t(\zeta)$ is a continuous function, then it follows
\begin{equation}\label{final}
\dim_H S \leq t(\log \Lambda).
\end{equation}
Recall the definitions of the sets $S_0$ and $S$, \eqref{S}, \eqref{s:0}. Our goal is to estimate the Hausdorff dimension of the set $S_0$. From \eqref{imp} we have
$$
\dim_H S_0 \leq \dim_H S \leq t(\log \Lambda). 
$$

\smallskip
To estimate the Hausdorff dimension of $S_0$ from below we notice, that
\begin{equation}\label{lower}
\{x\in[0,1]: a_k(x)\geq 2 C_0, \hbox{ for all } k\geq k_x\}\subset S_0.
\end{equation}
One gets this by considering the set of all $\{n_k,m_k,\epsilon_{k+1}\}_{k=1}^\infty$, where $\epsilon_{k+1}=1$, and $m_k,n_k\geq C_0$.
But according to Theorem 2 in \cite{Good}, the Hausdorff dimension of the set \eqref{lower} in the case $2C_0 \geq 20$, can be estimated from below as follows
$$
\frac{1}{2} + \frac{1}{2\log(2C_0+2)}\leq \dim_H S_0.
$$
Combining this with \eqref{final}, we get
\begin{equation}\label{fin1}
\frac{1}{2} + \frac{1}{2\log(2C_0+2)}\leq \dim_H S_0 \leq t(\log\Lambda).
\end{equation}
\end{proof}

\medskip

\section{Proof of main theorem}\label{sec:3}
\begin{proof}

We have
$$
A(\alpha,\beta)=\frac{1-\alpha}{1+\beta}.
$$
We want to estimate the Hausdorff dimension of the set $\mathcal{P}_0$. From \eqref{maincor} we have, that 
$$
\dim_H\mathcal{P}_0 \leq \dim_H\mathcal{P}.
$$
For this we refer to a standard fact from the theory of Hausdorff dimensions. If $F:G \subset \mathbb{R}^2\rightarrow \mathbb{R}$ is $F \in C^1$, $\triangledown F \neq 0$ for all $(\alpha,\beta)\in G$, then for any set $E \subset  \mathbb{R}$ one has
\begin{equation}\label{haus1}
\dim_HF^{-1}(E)=1 + \dim_HE.
\end{equation}
To verify this conditions for $A(\alpha,\beta)$ in $D$ we compute
$$
\frac{\partial}{\partial\alpha} A(\alpha,\beta)=-\frac{1}{1+\beta}
$$
and
$$
\frac{\partial}{\partial\beta} A(\alpha,\beta)=-\frac{1 - \alpha}{(1+\beta)^2}.
$$
We see, that $\triangledown F \neq 0$ in $D_0$. Hence, from formula \eqref{haus1}
$$
\dim_H (\mathcal{P}_0)=\dim_H(A^{-1}(S_0))=1 + \dim_H S_0.
$$
But then, from Proposition \ref{prop:3}
$$
\frac{3}{2} + \frac{1}{2\log(2C_0+2)} \leq \dim_H (\mathcal{P}_0) \leq 1 + t(\log\Lambda).
$$
Since, for $C_0 \geq 20$ we had $\log \Lambda > \log K_0 = \zeta_0$, then, in view of Theorem \ref{t-func}, we get
$$
t(\log\Lambda) < 1.
$$
Therefore
$$
\frac{3}{2}< \dim_H (\mathcal{P}_0) < 2.
$$
\end{proof}

\section*{Acknowledgements} The author would like to express his gratitude to Michael Benedicks for
his guidance and valuable suggestions and also to El Houcein El Abdalaoui for many useful discussions and comments on the manuscript.

\end{document}